\documentclass{amsart}
\usepackage{amscd}
\usepackage{amsmath,empheq}
\usepackage{amsfonts}
\usepackage{amssymb}
\usepackage{mathrsfs}
\usepackage[all]{xy}
\newtheorem{theorem}{Theorem}

\newtheorem{prop}[theorem]{Proposition}
\newtheorem{remark}{Remark}
\newtheorem{corollary}[theorem]{Corollary}
\newtheorem{claim}{Claim}

\newenvironment{proof-sketch}{\noindent{\bf Sketch of Proof}\hspace*{1em}}{\qed\bigskip}

\everymath{\displaystyle}
\newcommand{\RR}{\mathbb R}
\newcommand{\NN}{\mathbb N}

\newcommand{\ZZ}{\mathbb Z}

\renewcommand{\leq}{\leqslant}

\renewcommand{\geq}{\geqslant}
\begin{document}
\title[Asymmetric Robin problems]{Asymmetric Robin problems with indefinite potential and concave terms}
\author[N.S. Papageorgiou]{Nikolaos S. Papageorgiou}
\address[N.S. Papageorgiou]{National Technical University, Department of Mathematics,
				Zografou Campus, Athens 15780, Greece \& Institute of Mathematics, Physics and Mechanics, Jadranska 19, 1000 Ljubljana, Slovenia}
\email{\tt npapg@math.ntua.gr}
\author[V.D. R\u{a}dulescu]{Vicen\c{t}iu D. R\u{a}dulescu}
\address[V.D. R\u{a}dulescu]{ Institute of Mathematics, Physics and Mechanics, Jadranska 19, 1000 Ljubljana, Slovenia \& 
Faculty of Applied Mathematics, AGH University of Science and Technology, 30-059 Krak\'ow, Poland		
		 \& Institute of Mathematics ``Simion Stoilow" of the Romanian Academy, P.O. Box 1-764,  014700 Bucharest, Romania}
\email{\tt vicentiu.radulescu@imfm.si}
\author[D.D. Repov\v{s}]{Du\v{s}an D. Repov\v{s}}
\address[D.D. Repov\v{s}]{Faculty of Education and Faculty of Mathematics and Physics, University of Ljublijiana,  SI-1000 Ljubljana, Slovenia \& Institute of Mathematics, Physics and Mechanics, Jadranska 19, 1000 Ljubljana, Slovenia}
\email{\tt dusan.repovs@guest.arnes.si}
\begin{abstract}
We consider a parametric semilinear Robin problem driven by the Laplacian plus an indefinite and unbounded potential. In the reaction, we have the competing effects of a concave term appearing with a negative sign and of an asymmetric asymptotically linear term which is resonant in the negative direction. Using variational methods together with truncation and perturbation techniques and Morse theory (critical groups) we prove two multiplicity theorems producing four and five respectively nontrivial smooth solutions when the parameter $\lambda>0$ is small.
\end{abstract}
\maketitle

\section{Introduction}

Let $\Omega\subseteq\RR^N$ be a bounded domain with a $C^2$-boundary $\partial\Omega$. In this paper we study the following parametric Robin problem:
\begin{equation}\tag{$P_{\lambda}$}\label{eqP}
	\left\{
		\begin{array}{l}
			-\Delta u(z) + \xi(z)u(z) = f(z,u(z)) - \lambda\,|u(z)|^{q-2}u(z)\ \mbox{in}\ \Omega,\\
			\frac{\partial u}{\partial n} + \beta(z)u(z)=0\ \mbox{on}\ \partial\Omega.
		\end{array}
	\right\}
\end{equation}

In this problem, the potential function $\xi\in L^s(\Omega)$ ($s>N$) is indefinite (that is, sign changing). In the reaction (right-hand side), the function $f(z,x)$ is Carath\'eodory (that is, for all $x\in\RR$ the function $z\mapsto f(z,x)$ is measurable and for almost all $z\in\Omega$ the function $x\mapsto f(z,x)$ is continuous) and $f(z,\cdot)$ has linear growth near $\pm\infty$. However, the asymptotic behaviour of $f(z,\cdot)$ as $x\rightarrow\pm\infty$ is asymmetric. More precisely, we assume that the quotient $\frac{f(z,x)}{x}$ as $x\rightarrow+\infty$ stays the principal eigenvalue $\hat\lambda_1$ of the differential operator $u\mapsto -\Delta u+\xi(z)u$ with Robin boundary condition, while as $x\rightarrow-\infty$ the quotient $\frac{f(z,x)}{x}$ stays below $\hat\lambda_1$ with possible interaction (resonance) with respect to $\hat\lambda_1$ from the left. So, $f(z,\cdot)$ is a crossing (jumping) nonlinearity. In the term $-\lambda|u|^{q-2}u$, $ \lambda>0$ is a parameter and $1<q<2$. Hence this term is a concave nonlinearity. Therefore in the reaction we have the competing effects of resonant and concave terms. However, note that in our problem the concave nonlinearity enters with a negative sign. Such problems were considered by Perera \cite{12}, de Paiva \& Massa \cite{6}, de Paiva \& Presoto \cite{7} for Dirichlet problems with zero potential (that is, $\xi\equiv0$). Of the aforementioned works, only de Paiva \& Presoto \cite{7} have an asymmetric reaction of special form, which is superlinear in the positive direction and linear and nonresonant in the negative direction. Recently, problems with asymmetric reaction were studied by D'Agui, Marano \& Papageorgiou \cite{2} (Robin problems), Papageorgiou \& R\u{a}dulescu \cite{8, 11} (Neumann and Robin problems) and Recova \& Rumbos \cite{14} (Dirichlet problems).

We prove two multiplicity results in which we show that for all small $\lambda>0$ the problem has four and five nontrivial smooth solutions, respectively. Our approach uses variational tools based in the critical point theory, together with suitable truncation, perturbation and comparison techniques and Morse theory (critical groups).

\section{Mathematical background and hypotheses}

Let $X$ be a Banach space. We denote by $X^*$  the topological dual of $X$ and by $\langle\cdot,\cdot\rangle$ the duality brackets for the pair $(X^*,X)$. Given $\varphi\in C^1(X,\RR)$, we say that $\varphi$ satisfies the ``Cerami condition" (the ``C-condition" for short) if the following property holds:
$$
\begin{array}{c}
	``\mbox{Every sequence}\ \{u_n\}_{n\geq1}\subseteq X\ \mbox{such that}\\
	\{\varphi(u_n)\}_{n\geq1}\subseteq\RR\ \mbox{is bounded and}\ 
	(1+||u_n||)\varphi'(u_n)\rightarrow0\ \mbox{in}\ X^*\ \mbox{as}\ n\rightarrow\infty,\\
	\mbox{admits a strongly convergent subsequence}."
\end{array}
$$

This compactness-type condition on $\varphi(\cdot)$, is crucial in deriving the minimax theory of the critical values of $\varphi$. One of the main results in that theory is the so-called ``mountain pass theorem", which we recall below.

\begin{theorem}\label{th1}
	Assume that $\varphi\in C^1(X,\RR)$ satisfies the $C$-condition, $u_0,u_1\in X$, $||u_1-u_0||>r$,
	$$
	\max\{\varphi(u_0),\varphi(u_1)\} < \inf[\varphi(u): ||u-u_0||=r]=m_r
	$$
	and $c=\inf_{\gamma\in\Gamma}\max_{0\leq t\leq1}\varphi(\gamma(t))$ with $\Gamma=\{\gamma\in C([0,1],X):\gamma(0)=u_0, \gamma(1)=u_1\}$. Then $c\geq m_r$ and $c$ is a critical value of $\varphi$ (that is, there exists $u\in X$ such that $\varphi(u)=c, \varphi'(u)=0$).
\end{theorem}

Recall that a Banach space $X$ has the ``Kadec-Klee property", if the following holds:
$$
``u_n\xrightarrow{w}u\ \mbox{in}\ X,\ ||u_n||\rightarrow||u||\Rightarrow u_n\rightarrow u\ \mbox{in}\ X."
$$

It is an easy consequence of the parallelogram law, that every Hilbert space has the Kadec-Klee property (see Gasinski \& Papageorgiou \cite{3}).

In the study of problem (\ref{eqP}), we will use the following three spaces:
$$
H^1(\Omega),\ C^1(\overline{\Omega}),\ L^r(\partial\Omega)\ (1\leq r\leq\infty).
$$

The Sobolev space $H^1(\Omega)$ is a Hilbert space with inner product given by
$$
(u,h) = \int_\Omega(Du, Dh)_{\RR^N}dz + \int_\Omega uhdz\ \mbox{for all}\ u,h\in H^1(\Omega).
$$

We denote by $||\cdot||$ the corresponding norm on $H^1(\Omega)$. So, we have
$$
||u|| = \left[||u||^2_2 + ||Du||^2_2\right]^{1/2}\ \mbox{for all}\ u\in H^1(\Omega).
$$

The space $C^1(\overline\Omega)$ is an ordered Banach space with positive (order) cone
$$
C_+ = \{u\in C^1(\overline\Omega):u(z)\geq0\ \mbox{for all}\ z\in\overline\Omega\}.
$$

This cone has a nonempty interior. Note that
$$
D_+ = \{u\in C_+:u(z)>0\ \mbox{for all}\ z\in\overline\Omega\}\subseteq{\rm int}\, C_+.
$$

In fact, $D_+$ is the interior of $C_+$ when the latter is furnished with the relative $C(\overline\Omega)$-norm topology.

On $\partial\Omega$ we consider the $(N-1)$-dimensional Hausdorff (surface) measure $\sigma(\cdot)$. Using this measure on $\partial\Omega$, we can define in the usual way the ``boundary" Lebesgue spaces $L^r(\partial\Omega)$ (for $1\leq r\leq\infty$). From the theory of Sobolev spaces, we know that there exists a unique continuous linear map $\gamma_0:H^1 (\Omega)\rightarrow L^2(\partial\Omega)$ known as the ``trace map" such that
$$
\gamma_0(u)=u|_{\partial\Omega}\ \mbox{for all}\ u\in H^1(\Omega)\cap C^1(\overline\Omega).
$$

So, the trace map assigns ``boundary values" to every Sobolev function. The trace map is compact into $L^p(\partial\Omega)$ for all $1\leq p<\frac{2(N-1)}{N-2}$ if $N\geq3$ and into $L^p(\partial\Omega)$ for all $1\leq p\leq\infty$ if $N=1,2$. Also we have
$$
{\rm im}\,\gamma_0 = H^{\frac{1}{2},2}(\partial\Omega)\ \mbox{and}\ {\rm ker}\,\gamma_0=H^1_0(\Omega).
$$

In what follows, for the sake of notational simplicity, we drop the use of the trace map $\gamma_0$. All restrictions of Sobolev functions on $\partial\Omega$ are understood in the sense of traces.

Next, we consider the following linear eigenvalue problem:
\begin{equation}\label{eq1}
	\left\{
		\begin{array}{l}
			-\Delta u(z) + \xi(z)u(z) = \hat\lambda u(z)\ \mbox{in}\ \Omega,\\
			\frac{\partial u}{\partial n} + \beta(z)(u)=0\ \mbox{on}\ \partial\Omega.
		\end{array}
	\right\}
\end{equation}

This problem was studied by D'Agui, Marano \& Papageorgiou \cite{2}. We impose the following conditions on the potential function $\xi(\cdot)$ and on the boundary coefficient $\beta(\cdot)$.

\smallskip
$H(\xi):$ $\xi\in L^s(\Omega)$ with $s>N$.

\begin{remark}
	The potential function is both unbounded and sign-changing.
\end{remark}

$H(\beta):$ $\beta\in W^{1,\infty}(\partial\Omega)$ and $\beta(z)\geq0$ for all $z\in\partial\Omega$.

\begin{remark}
	If $\beta\equiv0$, then we recover the Neumann problem.
\end{remark}

Let $\gamma: H^1(\Omega)\rightarrow\RR$ be the $C^2$ functional defined by
$$
\gamma(u)=||Du||^2_2 + \int_\Omega\xi(z)u^2dz + \int_{\partial\Omega}\beta(z)u^2d\sigma\ \mbox{for all}\ u\in H^1(\Omega).
$$

Problem (\ref{eq1}) admits a smallest eigenvalue $\hat{\lambda}_1\in\RR$ given by
\begin{equation}\label{eq2}
	\hat{\lambda_1} = \inf\left\{\frac{\gamma(u)}{||u||^2_2}: u\in H^1(\Omega), \ u\neq0\right\}.
\end{equation}

Moreover, there exists $\mu>0$ such that
\begin{equation}\label{eq3}
	\gamma(u) + \mu||u||^2_2 \geq c_0||u||^2\ \mbox{for some}\ c_0>0,\ \mbox{for all}\ u\in H^1(\Omega).
\end{equation}

Using (\ref{eq3}) and the special theorem for compact self-adjoint operators on Hilbert spaces, we produce the full spectrum of (\ref{eq2}). This consists of a sequence $\{\hat{\lambda}_k\}_{k\in\NN}$ of distinct eigenvalues such that $\hat\lambda_k\rightarrow+\infty$. Let $E(\hat\lambda_k)$ denote the eigenspace corresponding to the eigenvalue $\hat\lambda_k$. From the regularity theory of Wang \cite{15}, we have
$$
E(\hat\lambda_k)\subseteq C^1(\overline\Omega)\ \mbox{for all}\ k\in\NN.
$$

Each eigenspace has the ``Unique Continuation Property" (UCP for short). This means that if $u\in E(\hat\lambda_k)$ vanishes on a set of positive Lebesgue measure, then $u\equiv0$.

Let $\overline{H}_m = \bigoplus^m_{k=1} E(\hat\lambda_k)$ and $\hat{H}_m = \overline{H}^\perp_m = \overline{\bigoplus_{k\geq m+1}E(\hat\lambda_k)}$. We have
$$
H^1(\Omega) = \overline{H}_m \oplus \hat{H}_m.
$$

Moreover, for every $m\geq2$, we  have variational characterizations for the eigenvalues for the eigenvalues $\hat\lambda_m$ analogus to that for $\hat\lambda_1$ (see (\ref{eq2})):
\begin{eqnarray}
	\hat\lambda_m = & \inf\left\{\frac{\gamma(u)}{||u||^2_2}: u\in\hat{H}_{m-1}, u\neq0\right\} \nonumber \\
	= & \sup \left\{\frac{\gamma(u)}{||u||^2_2}: u\in\overline{H}_m, u\neq0\right\}, m\geq2. \label{eq4}
\end{eqnarray}

In (\ref{eq2}) the infimum is realized on $E(\hat\lambda_1)$, while in (\ref{eq4}) both the infimum and the supremum are realized on $E(\hat\lambda_m)$. We know that $\dim E(\hat\lambda_1)=1$ (that is, the first eigenvalue $\hat\lambda_1$ is simple). Hence the elements of $E(\hat\lambda_1)$ have constant sign. We denote by $\hat{u}_1\in C_+\backslash\{0\}$  the positive $L^2$-normalized eigenfunction (that is, $||\hat{u}_1||_2=1$) corresponding to $\hat\lambda_1$. From the strong maximum principle we have $\hat{u}_1(z)>0$ for all $z\in\Omega$ and if $\xi^+\in L^\infty(\Omega)$ (that is, the potential function is bounded above), then by the Hopf boundary point theorem we have $\hat{u}_1\in D_+$ (see Pucci \& Serrin \cite[p. 120]{13}).

Using (\ref{eq2}), (\ref{eq4}) and the above properties, we have the following useful inequalities.

\begin{prop}\label{prop2}
	\begin{itemize}
		\item [(a)] If $\vartheta\in L^\infty(\Omega),\ \vartheta(z)\leq\hat\lambda_m$ for almost all $z\in\Omega$, $\vartheta\not\equiv\hat\lambda_m$ $(m\in\NN)$, then there exists $c_1>0$ such that
			$$
			c_1||u||^2\leq\gamma(u) - \int_\Omega\vartheta(z)u^2dz\ \mbox{for all}\ u\in\hat{H}_{m-1}.
			$$
		\item [(b)] If $\vartheta\in L^\infty(\Omega),\ \vartheta(z)\geq\hat\lambda_m$ for almost all $z\in\Omega$, $\vartheta\not\equiv\hat\lambda_m$ $(m\in\NN)$, then there exists $c_2>0$ such that
			$$
			\gamma(u) - \int_\Omega\vartheta(z)u^2dz\leq c_2||u||^2\ \mbox{for all}\ u\in\overline{H}_m.
			$$
	\end{itemize}
\end{prop}

Note that if $\xi\equiv0,\ \beta\equiv0$, then $\hat{\lambda}_1=0$, while if $\xi\geq0$ and either $\xi\equiv0$ or $\beta\equiv0$, then $\hat\lambda_1>0$. Also, the elements of $E(\hat\lambda_k)$ for $k\geq2$ are nodal (that is, sign-changing).

In addition to the eigenvalue problem (\ref{eq1}), we can consider a weighted version of it. So, let $m\in L^\infty(\Omega),\ m(z)\geq0$ for almost all $z\in\Omega,\ m\not\equiv0$ and consider the following linear eigenvalue problem
\begin{equation}\label{eq5}
	\left\{
		\begin{array}{ll}
			-\Delta u(z) + \xi(z)u(z) = \tilde\lambda m(z)u(z)\ \mbox{in}\ \Omega,\\
			\frac{\partial u}{\partial n} + \beta(z)u=0\ \mbox{on}\ \partial\Omega.
		\end{array}
	\right\}
\end{equation}

This eigenvalue problem exhibits the same properties as (\ref{eq1}). So, the spectrum consists of a sequence $\{\tilde\lambda_k(m)\}_{k\in\NN}$ of distinct eigenvalues such that $\tilde\lambda_k(m)\rightarrow+\infty$ as $k\rightarrow+\infty$. As for (\ref{eq1}), the first eigenvalue $\tilde\lambda_1(m)$ is simple and the elements of $E(\tilde\lambda_1(m))\subseteq C^1(\overline\Omega)$ have fixed sign, while the elements of $E(\tilde\lambda_k(m))\subseteq C^1(\overline\Omega)$ (for all $k\geq2$) are nodal. We have variational characterisations for all the eigenvalues as in (\ref{eq2}) and (\ref{eq4}) only now the Rayleigh quotient is $\frac{\gamma(u)}{\int_\Omega m(z)u^2dz}$. Moreover, the eigenspaces have the UCP property. These properties lead to the following monotonicity property for the map $m\mapsto\tilde\lambda_k(m)$, $k\in\NN$.

\begin{prop}\label{prop3}
	If $m_1,\,m_2\in L^\infty(\Omega)$, $0\leq m_1(z)\leq m_2(z)$ for almost all $z\in\Omega$, $m_1\not\equiv0$, $m_2\not\equiv m_1$, then $\tilde\lambda_k(m_2)<\tilde\lambda_k(m_1)$ for all $k\in\NN$.
\end{prop}

Let $f_0:\Omega\times\RR\rightarrow\RR$ be a Carath\'eodory function such that
$$
|f_0(z,x)\leq a_0(z)[1+|x|^{r-1}]|\ \mbox{for almost all}\ x\in\RR,
$$
with $a_0\in L^\infty(\Omega)$ and $1<r\leq2^*=\left\{\begin{array}{ll}\frac{2N}{N-2} & \mbox{if}\ N\geq3\\+\infty & \mbox{if}\ N=1,2\end{array}\right.$ (the critical Sobolev exponent).
Let $F_0(z,x)=\int^x_0f_0(z,s)ds$ and consider the $C^1$-functional $\varphi_0:H^1(\Omega)\rightarrow\RR$ defined by
$$
\varphi_0(u) = \frac{1}{2}\gamma(u) - \int_\Omega F_0(z,u)dz\ \mbox{for all}\ u\in H^1(\Omega).
$$

As in Papageorgiou \& R\u{a}dulescu \cite[Proposition 8]{10}, using the regularity theory of Wang \cite{15}, we have the following result.

\begin{prop}\label{prop4}
	Assume that $u_0\in H^1(\Omega)$ is a local $C^1(\overline\Omega)$-minimizer of $\varphi_0(\cdot)$, that is, there exists $\rho_1>0$ such that
	$$
	\varphi_0(u_0)\leq\varphi_0(u_0+h)\ \mbox{for all}\ h\in C^1(\overline\Omega), \ ||h||_{C^1(\overline\Omega)}\leq \rho_1.
	$$
	Then $u_0\in C^{1,\alpha}(\overline\Omega)$ with $0<\alpha<1$ and $u_0$ is also a local $H^1(\Omega)$-minimizer of $\varphi_0$, that is, there exists $\rho_2>0$ such that
	$$
	\varphi_0(u_0)\leq\varphi_0(u_0+h)\ \mbox{for all}\ h\in C^1(\overline\Omega),\ ||h||\leq \rho_2.
	$$
\end{prop}

Now we will recall some definitions and facts from Morse theory (critical groups). So, let $X$ be a Banach space, $\varphi\in C^1(X,\RR)$ and $c\in\RR$. We introduce the following sets:
$$
\begin{array}{ll}
	\varphi^c = \{u\in X:\varphi(u)\leq c\},\\
	K_\varphi = \{u\in X:\varphi'(u)=0\},\\
	K^c_\varphi = \{u\in K_\varphi:\varphi(u)=c\}.
\end{array}
$$

Given a topological pair $(Y_1,Y_2)$ such that $Y_2\subseteq Y_1\subseteq X$, for every $k\in\NN_0$, we denote by $H_k(Y_1,Y_2)$  the $k$th-relative singular homology group for the pair $(Y_1,Y_2)$ with integer coefficients. Suppose that $u\in K^c_\varphi$ is isolated. The critical groups of $\varphi$ at $u$ are defined by
$$
C_k(\varphi,u) = H_k(\varphi^c\cap U,\varphi^c\cap U\backslash\{u\})\ \mbox{for all}\ k\in\NN_0,
$$
with $U$ being a neighbourhood of $u$ such that $K_\varphi\cap\varphi^c\cap U=\{u\}$. The excision property of singular homology, implies that the above definition of critical groups is independent of the isolating neighbourhood $U$. If $u$ is a local minimizer of $\varphi$, then
$$
C_k(\varphi,u) = \delta_{k,0}\ZZ\ \mbox{for all}\ k\in\NN_0.
$$
Here, $\delta_{k,m}$ denotes the Kronecker symbol defined by
$$
\delta_{k,m}=\left\{
	\begin{array}{ll}
		1 & \mbox{if}\ k=m,\\
		0 & \mbox{if}\ k\neq m.
	\end{array}
\right.
$$

Next, let us fix our notation. If $x\in\RR$, we set $x^\pm=\max\{\pm x,0\}$. Then for $u\in W^{1,p}(\Omega)$, we define $u^\pm(\cdot)= u(\cdot)^\pm$. We know that
$$
u^\pm\in W^{1,p}(\Omega),\ u=u^+-u^-,\ |u|=u^++u^-.
$$

Given a measurable function $g:\Omega\times\RR\rightarrow\RR$ (for example, a Carath\'eodory function), we denote by $N_g(\cdot)$  the Nemitsky (superposition) map defined by
$$
N_g(u)(\cdot) = g(\cdot,u(\cdot))\ \mbox{for all}\ u\in W^{1,p}(\Omega).
$$

Also, $A\in\mathscr{L}(H^1(\Omega),H^1(\Omega)^*)$ is defined by
$$
\langle A(u),h\rangle = \int_\Omega(Du,Dh)_{\RR^N}dz\ \mbox{for all}\ u,h\in H^1(\Omega).
$$

The hypotheses on the nonlinearity $f(z,x)$, are the following:

\smallskip
$H(f):$ $f:\Omega\times\RR\rightarrow\RR$ is a Carath\'eodory function such that $f(z,0)=0$ for almost all $z\in\Omega$ and
\begin{itemize}
	\item [(i)] for every $\rho>0$, there exists $a_\rho\in L^\infty(\Omega)$ such that
		$$
		|f(z,x)|\leq a_\rho(z)\ \mbox{for almost all}\ z\in\Omega,\ \mbox{all}\ |x|\leq\rho;
		$$
	\item [(ii)] there exist functions $\eta,\hat{\eta}\in L^\infty(\Omega)$ and $m\in\NN,m\geq2$ such that
		$$
		\begin{array}{ll}
			\hat\lambda_1\leq\eta(z)\leq\hat\eta(z)\leq\hat\lambda_m\ \mbox{for almost all}\ z\in\Omega,\ \eta\not\equiv\hat\lambda_1,\ \hat\eta\not\equiv\hat\lambda_m,	\\
			\hat\eta(z)\leq\liminf_{x\rightarrow+\infty}\frac{f(z,x)}{x}\leq\limsup_{x\rightarrow+\infty}\frac{f(z,x)}{x}\leq\hat\eta(z)\ \mbox{uniformly for almost all}\ z\in\Omega
		\end{array}
		$$
		and there exists $\tilde\eta>0$ such that
		$$
		-\hat\eta\leq\liminf_{x\rightarrow-\infty}\frac{f(z,x)}{x}\leq\limsup_{x\rightarrow-\infty}\frac{f(z,x)}{x	}\leq\hat\lambda_1\ \mbox{uniformly for almost all}\ z\in\Omega;
		$$
	\item [(iii)] if $F(z,x)=\int^x_0 f(z,s)ds$, then
		$$
		\begin{array}{ll}
		f(z,x)x - 2F(z,x)\rightarrow+\infty\ \mbox{uniformly for almost all}\ z\in\Omega\ \mbox{as}\ x\rightarrow-\infty, \\	
		f(z,x)x - 2F(z,x)\geq0\ \mbox{for almost all}\ z\in\Omega,\ \mbox{all}\ x\geq M_0>0, \\
		F(z,x)\leq \frac{\hat\lambda_m}{2} x^2\ \mbox{for almost all}\ z\in\Omega,\ \mbox{all}\ x\in\RR;
		\end{array}
		$$
	\item [(iv)] there exist functions $\vartheta,\hat\vartheta\in L^\infty(\Omega)$ and $l\in\NN, l\geq m$ such that
		$$
		\begin{array}{ll}
			\hat\lambda_l\leq\vartheta(z)\leq\hat\vartheta(z)\leq\hat\lambda_{l+1}\ \mbox{for almost all}\ z\in\Omega,\ \vartheta\not\equiv\hat\lambda_l, \ \hat\vartheta\neq\hat\lambda_{l+1}, \\
			\vartheta(z)\leq\liminf_{x\rightarrow0}\frac{f(z,x)}{x}\leq\limsup_{x\rightarrow0}\frac{f(z,x)}{x}\leq\hat\vartheta(z)\ \mbox{uniformly for almost all}\ z\in\Omega.
		\end{array}
		$$
\end{itemize}

\begin{remark}
	Hypothesis $H(f)(ii)$ implies that $f(z,\cdot)$ has asymmetric behaviour as $x\rightarrow\pm\infty$ (jumping nonlinearity). Moreover, as $x\rightarrow-\infty$ we can have resonance with respect to the principal eigenvalue $\hat\lambda_1$. Hypothesis $H(f)(iii)$ implies that this resonance is from the left of $\hat\lambda_1$ in the sense that
	$$
	\hat\lambda_1 x^2-2F(z,x)\rightarrow+\infty\ \mbox{uniformly for almost all}\  z\in\Omega\ \mbox{as}\ x\rightarrow-\infty.
 	$$
\end{remark}

Note that hypotheses $H(f)(i),(ii),(iv)$ imply that
\begin{equation}\label{eq6}
	|f(z,x)|\leq c_3|x|\ \mbox{for almost all}\ z\in\Omega,\ \mbox{all}\ x\in\RR,\ \mbox{some}\ c_3>0.
\end{equation}

For every $\lambda>0$, let $\varphi_\lambda:H^1(\Omega)\rightarrow\RR$ be the energy functional for problem (\ref{eqP}) defined by
$$
\varphi_\lambda(u) = \frac{1}{2}\gamma(u) + \frac{\lambda}{q}||u||^q_q - \int_\Omega F(z,u)dz\ \mbox{for all}\ u\in H^1(\Omega).
$$

Evidently, $\varphi_\lambda\in C^1(H^1(\Omega), \RR)$.

Let $\mu>0$ be as in (\ref{eq3}). We introduce the following truncations-perturbations of the reaction in problem (\ref{eqP}):
\begin{equation}\label{eq7}
	\begin{array}{ll}
	k^+_\lambda(z,x)=\left\{
		\begin{array}{ll}
			0 & \mbox{if}\ x\leq0 \\
			f(z,x)-\lambda x^{q-1} + \mu x & \mbox{if}\ x>0
		\end{array}
	\right. \\
	k^-_\lambda(z,x)=\left\{
		\begin{array}{ll}
			f(z,x)-\lambda |x|^{q-2}x + \mu x & \mbox{if}\ x<0 \\
			0 & \mbox{if}\ x\geq0.
		\end{array}
	\right.		
	\end{array}
\end{equation}

Both are Carath\'eodory functions. We set $K^\pm_\lambda(z,x)=\int^x_0 k^\pm_\lambda(z,s)ds$ and consider the $C^1$-functionals $\hat\varphi^\pm_\lambda: H^1(\Omega)\rightarrow\RR$ defined by
$$
\hat\varphi^\pm_\lambda(u) = \frac{1}{2}\gamma(u) + \frac{\mu}{2}||u||^2_2 - \int_\Omega K^\pm_\lambda(z,u)dz\ \mbox{for all}\ u\in H^1(\Omega).
$$

\section{Compactness conditions for the functionals}

We consider the functionals $\hat\varphi^\pm_\lambda,\varphi_\lambda$ and we show that they satisfy the compactness-type condition

\begin{prop}\label{prop5}
	If hypotheses $H(\xi),\ H(\beta),\ H(f)$ hold, then for every $\lambda>0$ the functional $\hat\lambda^+_\lambda$ satisfies the C-condition.
\end{prop}

\begin{proof}
	We consider a sequence $\{u_n\}_{n\geq1}\subseteq H^1(\Omega)$ such that
	\begin{eqnarray}
		|\hat\varphi^+_\lambda(u_n)|\leq M_1\ \mbox{for some}\ M_1>0,\ \mbox{all}\ n\in\NN, \label{eq8} \\
		(1+||u_n||)(\hat\varphi^+_\lambda)'(u_n)\rightarrow0\ \mbox{in}\ H^1(\Omega)^*\ \mbox{as}\ n\rightarrow+\infty. \label{eq9}
	\end{eqnarray}
	
	From (\ref{eq9}) we have
	\begin{eqnarray}\label{eq10}
		&&|\langle A(u_n),h\rangle + \int_\Omega[\xi(z)+\mu]u_nhdz + \int_{\partial\Omega}\beta(z)u_nhd\sigma - \int_\Omega k^+_{\lambda}(z,u_n)hdz| \leq \frac{\epsilon_n||h||}{1+||u_n||} \\
		&&\mbox{for all}\ h\in H^1(\Omega)\ \mbox{with}\ \epsilon_n\rightarrow0^+.\nonumber
	\end{eqnarray}
	
	In (\ref{eq10}) we choose $h=-u_n^-\in H^1(\Omega)$. Then
	\begin{eqnarray}
		& \gamma(u_n^-) + \mu||u^-_n||^2_2 \leq \epsilon_n\ \mbox{for all}\ n\in\NN\ \mbox{(see (\ref{eq7}))}, \nonumber\\
		\Rightarrow & c_0||u^-_n||^2\leq \epsilon_n\ \mbox{for all}\ n\in\NN\ \mbox{(see (\ref{eq3}))}, \nonumber \\
		\Rightarrow & u^-_n\rightarrow0\ \mbox{in}\ H^1(\Omega)\ \mbox{as}\ n\rightarrow\infty. \label{eq11}
	\end{eqnarray}
	
	From(\ref{eq10}) and (\ref{eq11}), we have
	\begin{eqnarray}\label{eq12}
		&&|\langle A(u^+_n),h\rangle + \int_\Omega\xi(z)u^+_nhdz + \int_{\partial\Omega}\beta(z)u^+_nhd\sigma - \int_\Omega[f(z,u^+_n) - \lambda(u^+_n)^{q-1}]hdz| \leq \epsilon'_n||h|| \\
		&&\mbox{for all}\ h\in H^1(\Omega),\ \mbox{with}\ \epsilon_n'\rightarrow0^+\ \mbox{(see (\ref{eq7})).}\nonumber
	\end{eqnarray}
	
	We show that $\{u^+_n\}_{n\geq1}\subseteq H^1(\Omega)$ is bounded. Arguing by contradiction, suppose that
	\begin{equation}\label{eq13}
		||u^+_n||\rightarrow\infty\ \mbox{as}\ n\rightarrow\infty.
	\end{equation}
	
	Let $y_n=\frac{u^+_n}{||u^+_n||},\ n\in\NN$. Then $||y_n||=1,\ y_n\geq0$ for all $n\in\NN$. So, we may assume that
	\begin{equation}\label{eq14}
		y_n\xrightarrow{w}y\ \mbox{in}\ H^1(\Omega)\ \mbox{and}\ y_n\rightarrow y\ \mbox{in}\ L^2(\Omega)\ \mbox{and in}\ L^2(\partial\Omega),\ y\geq0.
	\end{equation}
	
	Using (\ref{eq12}) we obtain
	\begin{eqnarray}\label{eq15}
		&&|\langle A(y_n),h\rangle + \int_\Omega\xi(z)y_nhdz + \int_{\partial\Omega}\beta(z)y_nhd\sigma + \frac{\lambda}{||u^+_n||^{2-q}}\int_\Omega y^{q-1}_nhdz-\int_\Omega\frac{N_f(u^+_n)}{||u^+_n||}hdz|\nonumber\\
		&&\hspace{1cm}\leq\frac{\epsilon'||h||}{||u^+_n||}\ \mbox{for all}\ n\in\NN.
	\end{eqnarray}
	
	From (\ref{eq6}) we see that
	\begin{equation}\label{eq16}
		\left\{\frac{N_f(u^+_n)}{||u^+_n||}\right\}_{n\geq1}\subseteq L^2(\Omega)\ \mbox{is bounded}.
	\end{equation}
	
	So, by passing to a subsequence if necessary and using hypothesis $H(f)(ii)$, we have
	\begin{equation}\label{eq17}
		\frac{N_f(u^+_n)}{||u^+_n||}\xrightarrow{w}\nu (z)y\ \mbox{in}\ L^2(\Omega), \ \eta(z)\leq\nu(z)\leq\hat\eta(z)\ \mbox{for almost all}\ z\in\Omega
	\end{equation}
	(see Aizicovici, Papageorgiou \& Staicu \cite{1}, proof of Proposition 16).
	
	If in (\ref{eq15}) we choose $h=y_n-y\in H^1(\Omega)$, pass to the limit as $n\rightarrow\infty$ and use (\ref{eq13}), (\ref{eq14}), (\ref{eq16}) and the fact that $q<2$, then
	\begin{eqnarray}
		&& \lim_{n\rightarrow\infty}\langle A(y_n),y_n-y\rangle=0, \nonumber \\
		&\Rightarrow & ||Dy_n||_2\rightarrow||Dy||_2, \nonumber \\
		&\Rightarrow & y_n\rightarrow y\ \mbox{in}\ H^1(\Omega)\ \mbox{(by the Kadec-Klee property), hence}\ ||y||=1. \label{eq18}
	\end{eqnarray}
	
	In (\ref{eq15}) we pass to the limit as $n\rightarrow\infty$ and use (\ref{eq17}). We obtain
	\begin{eqnarray}\label{eq19}
		&& \langle A(y),h\rangle + \int_\Omega\xi(z)yhdz + \int_{\partial\Omega}\beta(z)yhd\sigma = \int_\Omega \nu (z)yhdz\ \mbox{for all}\ h\in H^1(\Omega), \nonumber \\
		&\Rightarrow & -\Delta y(z) + \xi(z)y(z)=\nu(z)y(z)\ \mbox{for almost all}\ z\in\Omega, \nonumber \\
		&& \frac{\partial y}{\partial n} + \beta(z)y=0\ \mbox{on}\ \partial\Omega\ \mbox{(see Papageorgiou \& R\u{a}dulescu \cite{9}).}
	\end{eqnarray}
	
	From (\ref{eq17}) and Proposition \ref{prop3}, we have
	\begin{equation}\label{eq20}
		\tilde\lambda_1(\nu)<\tilde\lambda_1(\hat\lambda_1)=1.
	\end{equation}
	
	Then (\ref{eq19}), (\ref{eq20}) and the fact that $||y||=1$ (see (\ref{eq18})) imply that
	$$
	y(\cdot)\ \mbox{must be nodal.}
	$$
	
	But this contradicts (\ref{eq14}). Therefore
	$$
		\begin{array}{ll}
			& \{u^+_n\}_{n\geq1}\subseteq H^1(\Omega)\ \mbox{is bounded}, \\
			\Rightarrow & \{u_n\}_{n\geq1}\subseteq H^1(\Omega)\ \mbox{is bounded (see (\ref{eq11}))}.
		\end{array}
	$$
	
	We may assume that
	\begin{equation}\label{eq21}
		u_n\xrightarrow{w}u\ \mbox{in}\ H^1(\Omega)\ \mbox{and}\ u_n\rightarrow u\ \mbox{in}\ L^2(\Omega)\ \mbox{and in}\ L^2(\partial\Omega).
	\end{equation}
	
	In (\ref{eq10}) we choose $h=u_n-u\in H^1(\Omega)$, pass to the limit as $n\rightarrow\infty$ and use (\ref{eq21}) and (\ref{eq6}). Then
	$$
		\begin{array}{ll}
			& \lim_{n\rightarrow\infty}\langle A(u_n),u_n-u\rangle=0, \\
			\Rightarrow & u_n\rightarrow u\ \mbox{in}\ H^1(\Omega)\ \mbox{(again by the Kadec-Klee property)}, \\
			\Rightarrow & \hat\varphi^+_\lambda\ \mbox{satisfies the C-condition}.
		\end{array}
	$$
The proof is now complete.
\end{proof}

\begin{prop}\label{prop6}
	If hypotheses $H(\xi),\ H(\beta),\ H(f)$ hold, then for every $\lambda>0$ the functional $\hat\varphi^-_\lambda$ is coercive.
\end{prop}

\begin{proof}
	According to hypothesis $H(f)(iii)$ given any $\rho>0$, we can find $M_2=M_2(\rho)>0$ such that
	\begin{equation}\label{eq22}
		\rho\leq f(z,x)x - 2F(z,x)\ \mbox{for almost all}\ z\in\Omega,\ \mbox{all}\ x\leq-M_2.
	\end{equation}
	
	We have
	\begin{eqnarray}
		& \frac{d}{dx}\left(\frac{F(z,x)}{x^2}\right) & = \frac{f(z,x)x^2 - 2xF(z,x)}{x^4} \nonumber \\
		& & = \frac{f(z,x)x - 2F(z,x)}{|x|^2x} \nonumber \\
		& & \leq \frac{\rho}{|x|^2x}\ \mbox{for almost all}\ z\in\Omega,\ \mbox{all}\ x\leq-M_2\ \mbox{(see (\ref{eq22}))}, \nonumber \\
		\Rightarrow & \frac{F(z,v)}{v^2}-\frac{F(z,y)}{y^2} & \geq \frac{\rho}{2}\left[\frac{1}{y^2}-\frac{1}{v^2}\right]\ \mbox{for almost all}\ z\in\Omega,\ \mbox{all}\ v\leq y\leq-M_2. \label{eq23}
	\end{eqnarray}
	
	From hypothesis $H(f)(ii)$ we have
	\begin{equation}\label{eq24}
		-\tilde\eta\leq\liminf_{x\rightarrow-\infty}\frac{2F(z,x)}{x^2}\leq\limsup_{x\rightarrow-\infty}\frac{2F(z,x)}{x^2}\leq\hat\lambda_1\ \mbox{uniformly for almost all}\ z\in\Omega.
	\end{equation}
	
	If in (\ref{eq23}) we let $v\rightarrow-\infty$ and use (\ref{eq24}), then
	\begin{eqnarray}\label{eq25}
		& \hat\lambda_1 y^2-2F(z,y)\geq\rho\ \mbox{for almost all}\ z\in\Omega,\ \mbox{all}\ y\leq-M_2, \nonumber \\
		\Rightarrow & \hat\lambda_1y^2 - 2F(z,y)\rightarrow+\infty\ \mbox{uniformly for almost all}\ z\in\Omega\ \mbox{as}\ y\rightarrow -\infty.
	\end{eqnarray}
	
	We proceed by contradiction and assume that $\hat\lambda^-_\lambda$ is not coercive. This means that we can find $\{u_n\}_{n\geq1}\subseteq H^1(\Omega)$ such that
	\begin{equation}\label{eq26}
		||u_n||\rightarrow\infty\ \mbox{as}\ n\rightarrow\infty\ \mbox{and}\ \hat\varphi^-_\lambda(u_n)\leq M_3\ \mbox{for some}\ M_3>0,\ \mbox{all}\ n\in\NN.
	\end{equation}
	
	Let $v_n=\frac{u_n}{||u_n||},\ n\in\NN$. Then $||v_n||=1$ for all $n\in\NN$ and so we may assume that
	\begin{equation}\label{eq27}
		v_n\xrightarrow{w} v\ \mbox{in}\ H^1(\Omega)\ \mbox{and}\ v_n\rightarrow v\ \mbox{in}\ L^2(\Omega)\ \mbox{and in}\ L^2(\partial\Omega).
	\end{equation}
	
	From (\ref{eq26}) we have
	\begin{eqnarray}
		& \frac{1}{2}\gamma(u_n) + \frac{\mu}{2}||u_n||^2_2 - \int_\Omega K^-_\lambda(z,u_n)dz\leq M_3\ \mbox{for all}\ n\in\NN, \nonumber \\
		\Rightarrow & \frac{1}{2}\gamma(v_n) + \frac{\mu}{2}||v_n||^2_2 - \int_\Omega\frac{K^-_\lambda(z,u_n)}{||u_n||^2}dz\leq\frac{M_3}{||u_n||^2}\ \mbox{for all}\ n\in\NN. \label{eq28}
	\end{eqnarray}
	
	From (\ref{eq6}) we obtain
	$$
	\begin{array}{ll}
		& |F(z,x)|\leq\frac{c_3}{2}|x^2|\ \mbox{for almost all}\ z\in\Omega,\ \mbox{all}\ x\in\RR, \\
		\Rightarrow & \left\{\frac{K^-_\lambda(\cdot,u_n(\cdot))}{||u_n||^2}\right\}_{n\geq1}\subseteq L^1(\Omega)\ \mbox{is uniformly integrable (see (\ref{eq7}) and (\ref{eq27}))}.
	\end{array}
	$$
	
	Hence by the Dunford-Pettis theorem and hypothesis $H(f)(ii)$ we have
	\begin{eqnarray}\label{eq29}
		&&\frac{K^-_\lambda(\cdot,u_n(\cdot))}{||u_n||^2}\xrightarrow{w}\frac{1}{2}[\tilde{e}(z)+\mu](v^-)^2\ \mbox{in}\ L^1(\Omega)\ \mbox{as}\ n\rightarrow\infty  \\
		&&\mbox{with}\ -\tilde\eta\leq\tilde{e}(z)\leq\hat\lambda_1\ \mbox{for almost all}\ z\in\Omega\ \mbox{(see \cite{1})}.	\nonumber		
		\end{eqnarray}

	We return to (\ref{eq28}), pass to the limit as $n\rightarrow\infty$ in (\ref{eq26}), (\ref{eq27}), (\ref{eq29}). Since $\gamma(\cdot)$ is sequentially weakly lower semicontinuous on $H^1(\Omega)$, we obtain
	\begin{eqnarray}\label{eq30}
		& \frac{1}{2}\gamma(v) + \frac{\mu}{2}||v||^2_2\leq\frac{1}{2}\int_\Omega[\tilde{e}(z)+\mu](v^-)^2dz \nonumber \\
		\Rightarrow & \gamma(v^-)\leq\int_\Omega\tilde{e}(z)(v^-)^2dz\ \mbox{(see (\ref{eq3}))}.
	\end{eqnarray}
	
	First we assume that $\tilde{e}\not\equiv\hat\lambda_1$ (see (\ref{eq29})). Then from (\ref{eq30}) and Proposition \ref{prop2} we have
	\begin{eqnarray}
		& c_1||v^-||^2\leq0, \nonumber \\
		\Rightarrow & v\geq0. \label{eq31}
	\end{eqnarray}
	
	Then on account of (\ref{eq27}) and (\ref{eq31}), we have
	\begin{equation}\label{eq32}
		v^-_n\xrightarrow{w}0\ \mbox{in}\ H^1(\Omega)\ \mbox{and}\ v^-_n\rightarrow0\ \mbox{in}\ L^2(\Omega)\ \mbox{and in}\ L^2(\partial\Omega).
	\end{equation}
	
	In (\ref{eq28}) we pass to the limit as $n\rightarrow\infty$ and use (\ref{eq32}), (\ref{eq30}) and the sequential weak lower semicontinuity of $\gamma(\cdot)$. We obtain
	$$
	\begin{array}{ll}
		& \gamma(v^+) + \mu||v^+|^2_2\leq0, \\
		\Rightarrow & c_0||v^+||^2\leq0\ \mbox{(see (\ref{eq3}))}, \\
		\Rightarrow & v=0\ \mbox{(see (\ref{eq31}))}.
	\end{array}
	$$
	
	From (\ref{eq28}) we obtain
	$$
	\begin{array}{ll}
		& ||Dv_n||_2\rightarrow0, \\
		\Rightarrow & v_n\rightarrow0\ \mbox{in}\ H^1(\Omega)\ \mbox{(see (\ref{eq27}))},
	\end{array}
	$$
	which contradicts the fact that $||v_n||=1$ for all $n\in\NN$.
	
	Next we assume that $\tilde{e}(z)=\hat\lambda_1$, for almost all $z\in\Omega$. From (\ref{eq30}) and (\ref{eq2}), we have
	\begin{eqnarray}
		& \gamma(v^-)=\hat\lambda_1||v^-||^2_2, \nonumber \\
		\Rightarrow & v^-=\tau\hat{u}_1\ \mbox{for some}\ \tau\geq0. \label{eq33}
	\end{eqnarray}
	
	If $\tau=0$, then $v\geq0$ and arguing as above (see the part of the proof after (\ref{eq31})), we obtain $v=0$, contradicting the fact that $||v_n||=1$ for all $n\in\NN$.
	
	If $\tau>0$, then from (\ref{eq33}) we have
	$$
	v(z)<0\ \mbox{for all}\ z\in\Omega.
	$$
	
	This means that
	$$
	\begin{array}{ll}
		& u^-_n(z)\rightarrow-\infty\ \mbox{for almost all}\ z\in\Omega\ \mbox{as}\ n\rightarrow\infty, \\
		\Rightarrow & \hat\lambda_1u^-_n(z)^2 - 2F(z,u^-_n(z))\rightarrow+\infty\ \mbox{for almost all}\ z\in\Omega\ \mbox{as}\ n\rightarrow\infty\ \mbox{(see (\ref{eq25}))} \\
		\Rightarrow & \int_\Omega[\hat\lambda_1(u^-_n)^2-2F(z,u^-_n)]dz\rightarrow+\infty\ \mbox{as}\ n\rightarrow\infty\ \mbox{(by Fatou's lemma, see (\ref{eq25}))}, \\
		\Rightarrow & \gamma(u^-_n)-2\int_\Omega F(z,-u^-_n)dz\rightarrow+\infty\ \mbox{as}\ n\rightarrow\infty\ \mbox{(see (\ref{eq2}))}, \\
		\Rightarrow & 2\hat\varphi^-_\lambda(u^-_n)\rightarrow+\infty\ \mbox{as}\ n\rightarrow\infty\ \mbox{(see (\ref{eq7}))}.
	\end{array}
	$$
	
	But this contradicts (\ref{eq26}).
	
	We conclude that $\hat\varphi^-_\lambda$ is coercive.
\end{proof}

This proposition leads to the following corollary (see Marano \& Papageorgiou \cite[Proposition 2.2]{4}).

\begin{corollary}\label{cor7}
	If hypotheses $H(\xi), H(\beta), H(f)$ hold, then for every $\lambda>0$ the functional $\hat\varphi^-_\lambda$ satisfies the C-condition.
\end{corollary}

Now we turn our attention to the energy functional $\varphi_\lambda,\ \lambda>0$.

\begin{prop}\label{prop8}
	If hypotheses $H(\xi), H(\beta), H(f)$ then for every $\lambda>0$ the functional $\varphi_\lambda$ satisfies the C-condition.
\end{prop}

\begin{proof}
	We consider a sequence $\{u_n\}_{n\geq1}\subseteq H^1(\Omega)$ such that
	\begin{eqnarray}
		|\varphi_\lambda(u_n)|\leq M_4\ \mbox{for some}\ M_4>0,\ \mbox{all}\ n\in\NN, \label{eq34} \\
		(1+||u_n||)\varphi'_\lambda(u_n)\rightarrow0\ \mbox{in}\ H^1(\Omega)^*\ \mbox{as}\ n\rightarrow\infty. \label{eq35}
	\end{eqnarray}
	
	From (\ref{eq35}) we have
	\begin{eqnarray}\label{eq36}
		|\langle A(u_n),h\rangle +\int_\Omega\xi(z)u_nhdz + \int_{\partial\Omega}\beta(z)u_nhd\sigma + \lambda\int_\Omega|u_n|^{q-2}u_nhd\sigma - \\
		\int_\Omega f(z,u_n)hdz| \leq \frac{\epsilon_n||h||}{1+||u_n||}\ \mbox{for all}\ h\in H^1(\Omega),\ \mbox{with}\ \epsilon_n\rightarrow0^+. \nonumber
	\end{eqnarray}
	
	In (\ref{eq36}) we choose $h=u_n\in H^1(\Omega)$. Then
	\begin{equation}\label{eq37}
		-\gamma(u_n) - \lambda||u_n||^q_q - \int_\Omega f(z,u_n)u_ndz\leq\epsilon_n\ \mbox{for all}\ n\in\NN.
	\end{equation}
	
	On the other hand, from (\ref{eq34}) we have
	\begin{equation}\label{eq38}
		\gamma(u_n) + \frac{2\lambda}{q}||u_n||^q_q - \int_\Omega 2F(z,u_n)dz \leq 2M_4\ \mbox{for all}\ n\in\NN.
	\end{equation}
	
	We add (\ref{eq37}) and (\ref{eq38}). Recalling that $q<2$, we obtain
	$$
	\int_\Omega[f(z,u_n)u_n - 2F(z,u_n)]dz\leq M_5\ \mbox{for all}\ n\in\NN.
	$$
	
	Using hypothesis $H(f)(iii)$, we see that
	\begin{equation}\label{eq39}
		\int_\Omega[f(z,-u^-_n)(-u^-_n) - 2F(z,-u^-_n)]dz\leq M_5\ \mbox{for all}\ n\in\NN.
	\end{equation}
	
	We use (\ref{eq39}) to show that $\{u^-_n\}_{n\geq1}\subseteq H^1(\Omega)$ is bounded. Arguing by contradiction, we may assume that
	\begin{equation}\label{eq40}
		||u^-_n||\rightarrow\infty\ \mbox{as}\ n\rightarrow\infty.
	\end{equation}
	
	Let $y_n=\frac{u^-_n}{||u^-_n||},\ n\in\NN$. Then $||y_n||=1,\ y_n\geq0$ for all $n\in\NN$. We may assume that
	\begin{equation}\label{eq41}
		y_n\xrightarrow{w}y\ \mbox{in}\ H^1(\Omega)\ \mbox{and}\ y_n\rightarrow y\ \mbox{in}\ L^2(\Omega)\ \mbox{and in}\ L^2(\partial\Omega),\ y\geq0.
	\end{equation}
	
	In (\ref{eq36}) we choose $h=-u^-_n\in H^1(\Omega)$. Then
	\begin{eqnarray}
		& \gamma(u^-_n) + \lambda||u^-_n||^q_q - \int_\Omega f(z,-u^-_n)(-u^-_n)dz \leq \epsilon_n\ \mbox{for all}\ n\in\NN, \nonumber \\
		\Rightarrow & \gamma(y_n) + \frac{\lambda}{||u^-_n||^{2-q}}||y_n||^q_q - \int_\Omega\frac{N_f(-u^-_n)}{||u^-_n||}y_ndz \leq \frac{\epsilon_n}{||u^-_n||^2}\ \mbox{for all}\ n\in\NN. \label{eq42}
	\end{eqnarray}
	
	From (\ref{eq6}) we see that
	$$
	\left\{\frac{N_f(-u^-_n)}{||u^-_n||}\right\}_{n\geq1}\subseteq L^2(\Omega)\ \mbox{is bounded}.
	$$
	
	So, by passing to a subsequence if necessary and using hypothesis $H(f)(ii)$ we have
	\begin{eqnarray}
		\frac{N_f(-u^-_n)}{||u^-_n||}\xrightarrow{w}\tilde{e}(z)y\ \mbox{in}\ L^2(\Omega)\ \mbox{as}\ n\rightarrow\infty \label{eq43} \\
		\mbox{with} -\tilde\eta\leq\tilde{e}(z)\leq\hat\lambda_1\ \mbox{for almost all}\ z\in\Omega. \nonumber
	\end{eqnarray}
	
	Returning to (\ref{eq42}), passing to the limit as $n\rightarrow\infty$ and using (\ref{eq40}) (recall that $q<2$), (\ref{eq41}), (\ref{eq43}) and the sequential weak lower semicontinuity of $\gamma(\cdot)$, we obtain
	\begin{equation}\label{eq44}
		\gamma(y)\leq\int_\Omega\tilde{e}(z)y^2dz.
	\end{equation}
	
	First we assume that $\tilde{e}\not\equiv\hat\lambda_1$ (see (\ref{eq43})). Then from (\ref{eq44}) and Proposition \ref{prop2}, we have
	$$
		\begin{array}{ll}
			& c_1||y||^2\leq0, \\
			\Rightarrow & y=0.	
		\end{array}
	$$
	
	From this and (\ref{eq42}), we infer that
	$$
		\begin{array}{ll}
			& ||Dy_n||_2\rightarrow0, \\
			\Rightarrow & y_n\rightarrow0\ \mbox{in}\ H^1(\Omega),
		\end{array}
	$$
	which contradicts the fact that $||y_n||=1$ for all $n\in\NN$.
	
	We now assume that $\tilde{e}(z)=\hat\lambda_1$ for almost all $z\in\Omega$. Then from (\ref{eq44}) and (\ref{eq2}) we have
	$$
		y=\tau\hat{u}_1\ \mbox{with}\ \tau\geq0.
	$$
	
	If $\tau=0$, then $y=0$ and as above we have
	$$
		y_n\rightarrow0\ \mbox{in}\ H^1(\Omega),
	$$
	a contradiction since $||y_n||=1$ for all $n\in\NN$.
	
	If $\tau>0$, then $y(z)>0$ for all $z\in\Omega$ and so
	$$
		\begin{array}{ll}
			& u^-_n(z)\rightarrow+\infty\ \mbox{for almost all}\ z\in\Omega,	 \\
			\Rightarrow & f(z,-u^-_n(z))(-u^-_n)(z) - 2F(z,-u^-_n(z))\rightarrow+\infty\ \mbox{for almost all}\ z\in\Omega\\ &\mbox{(see hypothesis $H(f)(iii)$),} \\
			\Rightarrow & \int_\Omega[f(z,-u^-_n)(-u^-_n) - 2F(z,-u^-_n)]dz\rightarrow+\infty\ \mbox{(by Fatou's lemma)}.
		\end{array}
	$$
	
	This contradicts (\ref{eq39}). Therefore
	\begin{equation}\label{eq45}
		\{u^-_n\}_{n\geq1}\subseteq H^1(\Omega)\ \mbox{is bounded}.
	\end{equation}
	
	Next, we show that $\{u^+_n\}_{n\geq1}\subseteq H^1(\Omega)$ is bounded. From (\ref{eq36}) and (\ref{eq45}), we have
	$$
		\begin{array}{rr}
			|\langle A(u^+_n),h\rangle + \int_\Omega\xi(z)u^+_nhdz + \int_{\partial\Omega}\beta(z)u^+_nhd\sigma + \lambda\int_\Omega(u^+_n)^{q-1}hdz - \int_\Omega f(z,u^+_n)hdz| \leq M_6	 \\
			\mbox{for some}\ M_6>0,\ \mbox{all}\ n\in\NN.
		\end{array}
	$$
	
	Using this bound and a contradiction argument as in the proof of Proposition \ref{prop5}, we show that
	$$
		\begin{array}{ll}
			& \{u^+_n\}_{n\geq1}\subseteq H^1(\Omega)\ \mbox{is bounded}, \\
			\Rightarrow & \{u_n\}_{n\geq1}\subseteq H^1(\Omega)\ \mbox{is bounded (see (\ref{eq45}))}.
		\end{array}
	$$
	
	From this, as before (see the proof of Proposition \ref{prop5}), via the Kadec-Klee property, we conclude that $\varphi_\lambda$ satisfies the C-condition.
\end{proof}

\section{Multiplicity theorems}

In this section using variational methods, truncation and perturbation techniques and Morse theory, we prove two multiplicity theorems for problem (\ref{eqP}) when $\lambda>0$ is small. In the first result, we produce four nontrivial smooth solutions, while in the second theorem, under stronger conditions on $f(z,\cdot)$, we establish the existence of five nontrivial smooth solutions.

We start with a result which allows us to satisfy the mountain-pass geometry (see Theorem \ref{th1}) and also distinguish the solutions we produce from the trivial one.

\begin{prop}\label{prop9}
	If hypotheses $H(\xi),\ H(\beta),\ H(f)$ hold, then for every $\lambda>0$, $u=0$ is a local minimizer of $\varphi_\lambda$ and of $\hat\varphi^\pm_\lambda$.
\end{prop}

\begin{proof}
	We do the proof for the functional $\varphi_\lambda$. The proofs for $\hat\varphi^\pm_\lambda$ are similar.
	
	Recall that
	\begin{equation}\label{eq46}
		|F(z,x)|\leq\frac{c_3}{2}|x|^2\ \mbox{for almost all}\ z\in\Omega,\ \mbox{all}\ x\in\RR\ \mbox{(see (\ref{eq6}))}.
	\end{equation}
	
	Then for $u\in C^1(\overline\Omega)\backslash\{0\}$ we have
	$$
	\begin{array}{ll}
		\varphi_\lambda(u) & \geq \frac{\lambda}{q}||u||^q_q - \left[\frac{c_8}{2}+||\xi||_\infty\right]||u||^2_2 \\
		& \mbox{(see (\ref{eq46}) and hypotheses $H(\xi), H(\beta)$)}. \\
		& \geq \frac{\lambda}{q}||u||^q_q - c_4||u||^{2-q}_\infty||u||^q_q\ \mbox{with}\ c_4=\left[\frac{c_1}{2}+||\xi||_\infty\right]>0 \\
		& = \left[\frac{\lambda}{q}-c_4||u||^{2-q}_\infty\right]||u||^q_q.
	\end{array}
	$$
	
	So, if $||u||_\infty\leq||u||_{C^1(\overline\Omega)}<\left(\frac{\lambda}{q c_4}\right)^\frac{1}{2-q}$, then $\varphi_\lambda(u)>0=\varphi_{\lambda}(0)$. Hence
	$$
	\begin{array}{ll}
		& u=0\ \mbox{is a local}\ C^1(\overline\Omega)-\mbox{minimizer of}\ \varphi_\lambda(\cdot), \\
		\Rightarrow & u=0\ \mbox{is a local}\ H^1(\Omega)-\mbox{minimizer of}\ \varphi_\lambda(\cdot)\ \mbox{(see Proposition \ref{prop4})}.
	\end{array}
	$$
	
	Similarly for the functionals $\hat{\varphi}^{\pm}_{\lambda}$.
\end{proof}

With the next proposition we guarantee that for  small $\lambda>0$  the functional $\hat\varphi^+_\lambda(\cdot)$ satisfies the mountain pass geometry (see Theorem \ref{th1}).

\begin{prop}\label{prop10}
	If hypotheses $H(\xi),\ H(\beta),\ H(f)$ hold, then we can find $\lambda^*>0$ such that for all $\lambda\in(0,\lambda^*)$, there is $t_0=t_0(\lambda)>0$ for which we have $\hat\varphi^+_\lambda(t_0\hat{u}_1)<0$.
\end{prop}

\begin{proof}
	Let $r>2$. From hypothesis $H(f)(iv)$ and (\ref{eq46}), we see that given $\epsilon>0$ we can find $c_5=c_5(\epsilon,r)>0$ such that
	\begin{equation}\label{eq47}
		F(z,x)\geq\frac{1}{2}[\vartheta(z)-\epsilon]x^2 - c_5x^r\ \mbox{for almost all}\ z\in\Omega,\ \mbox{all}\ x\geq0.
	\end{equation}
	
	Then for all $t>0$, we have
	\begin{eqnarray}
		&\hat\varphi^+_\lambda(t\hat{u}_1) & =\frac{t^2}{2}\gamma(\hat{u}_1) + \frac{\lambda t^q}{q}||\hat{u}_1||^q_q - \int_\Omega F(z,t\hat{u}_1)dz\ \mbox{(see (\ref{eq7}))} \nonumber \\
		& &\leq \frac{t^2}{2}\left[\gamma(\hat{u}_1) - \int_\Omega\vartheta(z)\hat{u}^2_1dz + \epsilon\right] + \frac{\lambda t^q}{q}||\hat{u}_1||^q_q + c_5t^r||\hat{u}_1||^r_r\nonumber\\
		&&(\mbox{see (\ref{eq47}) and recall that}\ ||\hat{u}_1||_2=1)\nonumber\\
		&&=\frac{t^2}{2}\left[\int_{\Omega}(\hat{\lambda}_1-\vartheta(z))\hat{u}_1^2dz+\epsilon\right]+\frac{\lambda t^q}{q}||\hat{u}_1||^q_q+c_5t^r||\hat{u}_1||^r_r.\label{eq48}
	\end{eqnarray}
	
	Note that
	$$
	k_* = \int_\Omega(\vartheta(z) - \hat\lambda_1)\hat{u}^2_1dz > 0\ \mbox{(see hypothesis $H(f)(iv)$)}.
	$$
	
	Choosing $\epsilon\in(0,k_*)$, we see from (\ref{eq48})  that
	\begin{eqnarray}
		\hat\varphi^+_\lambda(t\hat{u}_1) & \leq -c_6 t^2 + \lambda c_7 t^q + c_8t^r\ \mbox{for some}\ c_6, c_7, c_8>0 \nonumber \\
		& = [-c_6 + \lambda c_7t^{q-2} + c_8t^{r-2}]t^2. \label{eq49}
	\end{eqnarray}
	
	Consider the function
	$$
	\mathcal{J}_\lambda(t) = \lambda c_7 t^{q-2} + c_8t^{r-2}\ \mbox{for all}\ t>0.
	$$
	
	Evidently, $\mathcal{J}_\lambda\in C^1(0,+\infty)$ and since $1<q<2<r$, we see that
	$$
	\mathcal{J}_\lambda(t)\rightarrow+\infty\ \mbox{as}\ t\rightarrow0^+\ \mbox{and as}\ t\rightarrow+\infty.
	$$
	
	So, we can find $t_0\in(0,+\infty)$ such that
	\begin{eqnarray*}
		& &\mathcal{J}_\lambda(t_0) = \min[\mathcal{J}(t): 0<t<+\infty], \\
		&\Rightarrow & \mathcal{J}'_\lambda(t_0)=0, \\
		&\Rightarrow & \lambda c_7(2-q)t^{q-3}_0 = c_8(r-2)t^{r-3}_0, \\
		&\Rightarrow & t_0 = t_0(\lambda) = \left[\frac{\lambda c_7(2-q)}{c_8(r-2)}\right]^{\frac{1}{r-q}}.
	\end{eqnarray*}
	
	Then
	$$
	\mathcal{J}_\lambda(t_0) = \lambda c_7\frac{[c_8(r-2)]^\frac{2-q}{r-q}}{[\lambda c_2(2-q)]^\frac{2-q}{r-q}} + c_8\frac{[\lambda c_2(2-q)]^\frac{r-2}{2-q}}{[c_8(r-2)]^\frac{r-2}{2-q}}.
	$$
	
	Since $\frac{2-q}{r-q}<1$, we see that
	$$
	\mathcal{J}_\lambda(t_0)\rightarrow0^+\ \mbox{as}\ \lambda\rightarrow 0^+.
	$$
	
	So, we can find $\lambda^*>0$ such that
	$$
	\mathcal{J}_\lambda(t_0)<c_6\ \mbox{for all}\ \lambda\in(0,\lambda^*).
	$$
	
	Then from (\ref{eq49}) it follows that
	$$
	\hat\varphi^+_\lambda(t_0\hat{u}_1)<0\ \mbox{for all}\ \lambda\in(0,\lambda^*).
	$$
This completes the proof of Proposition \ref{prop10}.
\end{proof}

\begin{remark}
	In fact, a careful reading of the above proof reveals that
	\begin{equation}\label{eq50}
		\hat\varphi^-_\lambda(-t_0\hat{u}_1)<0\ \mbox{for all}\ \lambda\in(0,\lambda^*).
	\end{equation}
\end{remark}

\begin{prop}\label{prop11}
	If hypotheses $H(\xi),\ H(\beta),\ H(f)$ hold and $\lambda\in(0,\lambda^*)$, then there exists $u_0\in C^1(\overline\Omega)$ with $u_0(z)<0$ for all $z\in\Omega$ and
	$$
	\hat\varphi^-_\lambda(u_0) = \inf\left[\hat\varphi^-_\lambda(u): u\in H^1(\Omega)\right] < 0.
	$$
\end{prop}

\begin{proof}
	From Proposition \ref{prop6} we know that $\hat\varphi^-_\lambda$ is coercive. Also, the Sobolev embedding theorem and the compactness of the trace map, imply that $\hat\varphi^-_\lambda$ is sequentially weakly lower semicontinuous. Hence, by the Weierstrass-Tonelli theorem, we can find $u_0\in H^1(\Omega)$ such that
	\begin{equation}\label{eq51}
		\hat\varphi^-_\lambda(u_0)=\inf[\hat\varphi^-_\lambda(u):u\in W^{1,p}(\Omega)].
	\end{equation}
	
	From (\ref{eq50}) we see that
	$$
	\begin{array}{ll}
		& \hat\varphi^-_\lambda(u_0)<0=\hat\lambda^-_\lambda(0), \\
		\Rightarrow & u_0\neq0.
	\end{array}
	$$
	
	From (\ref{eq51}) we have
	\begin{eqnarray}
		&& (\hat\varphi^-_\lambda)'(u_0)=0, \nonumber \\
		&\Rightarrow & \langle A(u_0),h\rangle + \int_\Omega[\xi(z)+\mu]u_0hdz + \int_{\partial\Omega}\beta(z)u_0hd\sigma = \int_\Omega k^-_\lambda(z,u_0)hdz \label{eq52} \\
		&& \mbox{for all}\ h\in H^1(\Omega). \nonumber
	\end{eqnarray}
	
	In (\ref{eq52}) we choose $h=u^+_0\in H^1(\Omega)$. Then
	$$
	\begin{array}{ll}
		& \gamma(u^+_0) + \mu||u^+_0||^2_2 = 0\ \mbox{(see (\ref{eq7}))}, \\
		\Rightarrow & c_0||u^+_0||^2\leq0\ \mbox{(see (\ref{eq3}))}, \\
		\Rightarrow & u_0\leq0,\ u_0\neq0.
	\end{array}
	$$
	
	From (\ref{eq52}) and (\ref{eq7}) it follows that
	\begin{eqnarray}
		& \langle A(u_0),h\rangle + \int_\Omega\xi(z)u_0hdz + \int_{\partial\Omega}\beta(z)u_0hd\sigma = \int_\Omega[f(z,u_0) - \lambda|u_0|^{q-2}u_0]hdz \nonumber \\
		& \mbox{for all}\ h\in H^1(\Omega), \nonumber \\
		\Rightarrow & -\Delta u_0(z) + \xi(z)u_0(z) = f(z,u_0(z)) - \lambda|u_0(z)|^{q-2}u_0(z)\ \mbox{for almost all}\ z\in\Omega, \nonumber \\
		& \frac{\partial u_0}{\partial n} + \beta(z) u_0=0\ \mbox{on}\ \partial\Omega\ \mbox{(see Papageorgiou \& R\u{a}dulescu \cite{9}).} \label{eq53}
	\end{eqnarray}
	
	Let $\tau_\lambda(z,x)=f(z,x)-\lambda|x|^{q-2}x$ and $\hat{k}_\lambda(z)=\frac{\tau_\lambda(z,u_0(z))}{1+|u_0(z)|}$, for $\lambda>0$. Hypotheses $H(f)(i),(ii)$ imply that
	$$
	\begin{array}{ll}
		& |\tau_\lambda(z,x)|\leq c_9[1+|x|]\ \mbox{for almost all}\ z\in\Omega,\ \mbox{all}\ x\in\RR,\ \mbox{with}\ c_9=c_9(\lambda)>0, \\
		\Rightarrow & |\hat{k}_\lambda(z)| = \frac{|\tau_\lambda(z,u_0(z))|}{1+|u_0(z)|} \leq c_9\ \mbox{for almost all}\ z\in\Omega, \\
		\Rightarrow & \hat{k}_\lambda\in L^\infty(\Omega).
	\end{array}
	$$
	
	From (\ref{eq53}) we have
	$$
	\begin{array}{ll}
		-\Delta u_0(z) = [\xi(z)-\hat{k}_\lambda(z)]u_0(z) + \hat{k}_\lambda(z)\ \mbox{for almost all}\ z\in\Omega, \\
		\frac{\partial u_0}{\partial n} + \beta(z)u_0=0\ \mbox{on}\ \partial\Omega
	\end{array}
	$$
	(recall that $u_0\leq0$). Since $(\xi-\hat{k}_\lambda)(\cdot)\in L^s(\Omega)$ (for $s>N$), we deduce by Lemma 5.1 of Wang \cite{15} that
	$$
	u_0\in L^\infty(\Omega).
	$$
	
	Then the Calderon-Zygmund estimates (see Wang \cite[Lemma 5.2]{15}) imply that
	$$u_0\in(-C_+)\backslash\{0\}.$$
	
	Moreover, the Harnack inequality (see Theorem 7.2.1 in Pucci \& Serrin \cite[p. 163]{13}), implies that
	$$
	u_0(z)<0\ \mbox{for all}\ z\in\Omega.
	$$
This completes the proof.
\end{proof}

\begin{remark}
	The negative sign of the concave term does not allow us to conclude that $u_0\in-D_+$ when $\xi^+\in L^\infty(\Omega)$ (by Hopf's boundary point theorem, see Pucci \& Serrin \cite[p. 120]{13}).
\end{remark}

Now we can state and prove our first multiplicity theorem.

\begin{theorem}\label{th12}
	Assume that hypotheses $H(\xi),\ H(\beta),\ H(f)$ hold. Then there exists $\hat\lambda>0$ such that for all $\lambda\in(0,\hat\lambda)$ problem (\ref{eqP}) has at least four nontrivial solutions
	$$
	\begin{array}{ll}
		u_0,\hat{u}\in(-C_+)\backslash\{0\},\ u_0(z),\,\hat{u}(z)<0\ \mbox{for all}\ z\in\Omega, \\
		v_0\in C_+\backslash\{0\},\ v_0(z)>0\ \mbox{for all}\ z\in\Omega, \\
		y_0\in C^1(\overline\Omega)\backslash\{0\}.
	\end{array}
	$$
\end{theorem}

\begin{proof}
	From Proposition \ref{prop11} and its proof (see (\ref{eq53})), we already have one solution
	$$
	u_0\in(-C_+)\backslash\{0\},\ u_0(z)<0\ \mbox{for all}\ z\in\Omega,\ \mbox{when}\ \lambda\in(0,\lambda^*).
	$$
	
	This solution is a global minimizer of the functional $\hat\varphi^-_\lambda$.
	
	\begin{claim}\label{claim1}
		$u_0$ is a local minimizer of the energy functional $\varphi_\lambda$.
	\end{claim}
	
	We first show that $u_0$ is a local $C^1(\overline\Omega)$-minimizer of $\varphi_\lambda$. Arguing by contradiction, suppose that we could find a sequence $\{u_n\}_{n\geq1}\subseteq C^1(\overline\Omega)$ such that
	\begin{equation}\label{eq54}
		u_n\rightarrow u_0\ \mbox{in}\ C^1(\overline\Omega)\ \mbox{as}\ n\rightarrow\infty\ \mbox{and}\ \varphi_\lambda(u_n)<\varphi_\lambda(u_0)\ \mbox{for all}\ n\in\NN.
	\end{equation}
	
	Then for all $n\in\NN$, we have
	\begin{eqnarray}\label{eq55}
		&0 & > \varphi_\lambda(u_n)-\varphi_\lambda(u_0)\nonumber \\
		&& = \varphi_\lambda(u_n)-\hat\varphi^-_\lambda(u_0)\ (\mbox{since}\ \varphi_\lambda|_{(-C_+)}=\hat\varphi^-_\lambda|_{(-C_+)},\ \mbox{see (\ref{eq7})}) \nonumber\\
		&&\geq \varphi_\lambda(u_n)-\hat\varphi^-_\lambda(u_n)\ \mbox{(recall that $u_0$ is a global minimizer of $\hat\varphi^-_\lambda$)} \nonumber\\
		&&= \frac{1}{2}\gamma(u_n) + \frac{\lambda}{q}||u_n||^q_q - \int_\Omega F(z,u_n)dz - \frac{1}{2}\gamma(u_n) - \frac{\mu}{2}||u^+_n||^2_2 - \frac{\lambda}{q}||u^-_n||^q_q + \nonumber\\
		&&\int_\Omega F(z,-u^-_n)dz\ \mbox{(see (\ref{eq7}))} \nonumber\\
		&&= \frac{\lambda}{q}||u^+_n||^q_q - \frac{\mu}{2}||u^+_n||^2_2 - \int_\Omega F(z,u^+_n)dz \nonumber\\
		&&\geq \frac{\lambda}{q}||u^+_n||^q_q - \left(\frac{\mu+c_3}{2}\right)||u^+_n||^2_2\ (\mbox{see (\ref{eq46})})\nonumber\\
		&&\geq \frac{\lambda}{q}||u^+_n||^q_q-c_{10}||u^+_n||^{2-q}_{\infty}||u^+_n||q_q\ \mbox{with}\ c_{10}=\frac{\mu+c_3}{2}>0 \nonumber\\
		&& = \left[\frac{\lambda}{q}-c_{10}||u^+_n||^{2-q}_\infty\right]||u^+_n||^q_q.
	\end{eqnarray}
	
	From (\ref{eq54}) we have
	$$
	u^+_n\rightarrow0\ \mbox{in}\ C^1(\overline\Omega)\ \mbox{(recall that $u_0|_\Omega<0$)}.
	$$
	
	Therefore we can find $n_0\in\NN$ such that
	$$
	\begin{array}{ll}
		& \frac{\lambda}{q}>c_{10}||u^+_n||^{2-q}_\infty\ \mbox{for all}\ n\geq n_0, \\
		\Rightarrow & 0>\varphi_\lambda(u_n)-\varphi(u_0)>0\ \mbox{for all}\ n\geq n_0\ \mbox{(see (\ref{eq55}))}, \\
	\end{array}
	$$
	a contradiction.
	
	Hence we have that
	$$
	\begin{array}{ll}
		& u_0\ \mbox{is a local}\ C^1(\overline\Omega)-\mbox{minimizer of}\ \varphi_\lambda, \\
		\Rightarrow & u_0\ \mbox{is a local}\ H^1(\Omega)-\mbox{minimizer of}\ \varphi_\lambda\ \mbox{(see Proposition \ref{prop4})}.
	\end{array}
	$$
	
	This proves the claim.
	
	Using (\ref{eq7}) and the regularity theory of Wang \cite{15}, we can see that
	\begin{equation}\label{eq56}
		K_{\hat\lambda^-_\lambda}\subseteq(-C_+)\ \mbox{and}\ K_{\hat\varphi^+_\lambda}\subseteq C_+\ \mbox{for all}\ \lambda>0.
	\end{equation}
	
	On account of (\ref{eq56}) we see that we may assume that both critical sets $K_{\hat\varphi^-_\lambda}$ and $K_{\hat\varphi^+_\lambda}$ are finite or, otherwise, we already have an infinity of nontrivial smooth solutions of constant sign and so we are done.
	
	From Proposition \ref{prop9} we know that for all $\lambda>0, u=0$ is a local minimizer of $\hat\varphi^-_\lambda$. Since $K_{\hat\varphi^-_\lambda}$ is finite, we can find $\rho\in(0,||u_0||)$ small such that
	\begin{equation}\label{eq57}
		\hat\varphi^-_\lambda(u_0)<0=\hat\varphi^-_\lambda(0)<\inf\left[\hat\varphi^-_\lambda(u):||u||=\rho\right]=\hat{m}^-_\rho
	\end{equation}
	(see Aizicovici, Papageorgiou \& Staicu \cite{1}, proof of Proposition 29).
	
	From Corollary \ref{cor7} we know that
	\begin{equation}\label{eq58}
		\hat\varphi^-_\lambda\ \mbox{satisfies the C-condition}.
	\end{equation}
	
	Then (\ref{eq57}) and (\ref{eq58}) permit the use of Theorem \ref{th1} (the mountain pass theorem). So, we can find $\hat{u}\in H^1(\Omega)$ such that
	$$
	\hat{u}\in K_{\hat\varphi^-_\lambda}\subseteq(-C_+)\ \mbox{(see (\ref{eq56})) and}\ \hat\varphi^-_\lambda(u_0)<0=\hat\varphi^-_\lambda(0)<\hat{m}^-_\rho\leq\hat\varphi^-_\lambda(\hat{u}).
	$$
	
	It follows that
	$$
	\hat{u}\in(-C_+)\backslash\{0,u_0\}\ \mbox{is a solution of (\ref{eqP}) (see (\ref{eq7}))}.
	$$
	
	As before, Harnack's inequality implies that
	$$
	\hat{u}(z)<0\ \mbox{for all}\ z\in\Omega.
	$$
	
	Now we use once more Proposition \ref{prop9} to find $\rho_0\in(0,t_0)$ small enough such that
	\begin{equation}\label{eq59}
		0=\hat\varphi^+_\lambda(0)<\inf\left[\hat\varphi^+_\lambda(u):||u||=\rho_0\right]=\hat{m}^+_{\rho_0}, \lambda>0.
	\end{equation}
	
	Proposition \ref{prop10} implies that we can find $\lambda^*>0$ such that
	\begin{equation}\label{eq60}
		\hat\varphi^+_\lambda(t_0\hat{u}_1)<0\ \mbox{for all}\ \lambda\in(0,\lambda^*)\ \mbox{with}\ t_0=t_0(\lambda)>0.
	\end{equation}
	
	Moreover, Proposition \ref{eq5} implies that
	\begin{equation}\label{eq61}
		\hat\varphi^+_\lambda\ \mbox{satisfies the C-condition for all}\ \lambda>0.
	\end{equation}
	
	Then on account of (\ref{eq59}), (\ref{eq60}), (\ref{eq61}), we can apply Theorem \ref{th1} (the mountain pass theorem) and produce $v_0\in H^1(\Omega)$ such that
	$$
	\begin{array}{ll}
		& v_0\in K_{\hat\varphi^+_\lambda}\subseteq C_+\ \mbox{(see (\ref{eq56})) and }0=\hat\varphi^+_\lambda(0)<\hat{m}^+_\rho\leq\hat{\varphi}^+_\lambda(v_0), \\
		\Rightarrow & v_0\in C_+\backslash\{0\}\ \mbox{is a solution of (\ref{eqP})},\ \lambda\in(0,\lambda^*)\ \mbox{(see (\ref{eq7}))}.
	\end{array}
	$$
	
	Once again, Harnack's inequality guarantees that
	$$
	v_0(z)>0\ \mbox{for all}\ z\in\Omega.
	$$
	
	Let $l\in\NN$ be as in hypothesis $H(f)(iv)$ and set
	$$
	\overline{H}_l=\bigoplus^l_{k=1} E(\hat\lambda_k), \ \hat{H}_l=\overline{H}^\perp_l=\overline{\bigoplus_{k\geq l+1} E(\hat\lambda_k)}.
	$$
	
	We have
	$$
	H^1(\Omega)=\overline{H}_l\oplus\hat{H}_l\ \mbox{and}\ \dim\overline{H}_l<+\infty.
	$$
	
	Consider $u\in\overline{H}_l$. We have
	$$
	\begin{array}{ll}
		& \varphi_\lambda(u) = \frac{1}{2}\gamma(u) + \frac{\lambda}{q}||u||^q_q - \int_\Omega F(z,u)dz \\
		& \leq \frac{1}{2}\left[\gamma(u) - \int_\Omega\vartheta(z)u^2dz + \epsilon||u||^2\right] + c_{11}\left[\lambda||u||^q + ||u||^r\right] \\
		& \mbox{for some}\ c_{11}>0\ \mbox{(see (\ref{eq47}) and recall that all norms on $\overline{H}_l$ are equivalent)} \\
		& \leq \frac{1}{2}[-c_2+\epsilon]|||u||^2 + c_{11}[\lambda||u||^q+||u||^r]
	\end{array}
	$$
	(see Proposition \ref{prop2}). Choosing $\epsilon\in(0,c_2)$ we have
	$$
	\varphi_\lambda(u)\leq[-c_{12}+\lambda c_{11}||u||^{q-2} + c_{11}||u||^{r-2}]||u||^2\ \mbox{for some}\ c_{12}>0.
	$$
	
	Reasoning as in the proof of Proposition \ref{prop11}, we can find $\hat\lambda\in(0,\lambda^*]$ such that for all $\lambda\in(0,\hat\lambda]$ there exists $\rho_\lambda>0$ for which we have
	\begin{equation}\label{eq62}
		\varphi_\lambda(u)<0\ \mbox{for all}\ u\in\overline{H}_l,\ ||u||=\rho_\lambda.
	\end{equation}
	
	For $u\in\hat{H}_l$ we have
	\begin{eqnarray}\label{eq63}
	\left.\begin{aligned}
		\varphi_\lambda(u) & \geq \frac{1}{2}\gamma(u) + \frac{\lambda}{q}||u||^q_q - \frac{\hat\lambda m}{2}||u||^2_2\ \mbox{(see hypothesis $H(f)(iii)$)} \\
		& \geq \frac{1}{2}[\gamma(u)-\hat\lambda_l||u||^2_2] + \frac{\lambda}{q}||u||^q_q\ \mbox{(since $l\geq m$)} \\
		& \geq0.
	\end{aligned}\right.
	\end{eqnarray}
	
	Finally, consider the half-space
	$$
	H_+=\{t\hat{u}_1+\tilde{u}:t\geq0,\ \tilde{u}\in\hat{H}_l\}.
	$$
	
	Exploiting the orthogonality of $\hat{H}_l$ and $\overline{H}_l$, for every $u\in H_+$, we have
	\begin{eqnarray}\label{eq64}
		\left.\begin{aligned}
			\varphi_\lambda(u) & \geq\frac{1}{2}[t^2\gamma(\hat{u}_1)+\gamma(\tilde{u})] - \frac{\hat\lambda_m}{2}[t^2||\hat{u}_1||^2_2 + ||\tilde{u}||^2_2]\ \mbox{(see hypothesis $H(f)(iii)$)} \\
			& \geq 0\ \mbox{(since $\tilde{u}\in\hat{H}_l,\ l\geq m$).}
		\end{aligned}\right.
	\end{eqnarray}
	
	Then (\ref{eq62}), (\ref{eq63}), (\ref{eq64}) permit the use of Theorem 3.1 of Perera \cite{12}. So, we can find $y_0\in H^1(\Omega)$ such that
	\begin{eqnarray}\label{eq65}
		&&y_0\in K_{\varphi_\lambda}\subseteq C^1(\overline\Omega)\ \mbox{(by the regularity theory of Wang \cite{15})}, \nonumber\\
		&&\varphi_\lambda(y_0)<0=\varphi_\lambda(0)\ \mbox{and}\ C_{d_l-1}(\varphi_\lambda,y_0)\neq0\ \ (d_l=\dim\overline{H}_l).
	\end{eqnarray}
	
	From (\ref{eq65}) it is clear that $y_0\neq0$. Recall that
	$$
	0<\varphi_\lambda(\hat{u}), \varphi_\lambda(v_0)\ \mbox{(since $\varphi_{\lambda}=\left.\hat\varphi^-_{\lambda}\right|_{(-C_+)}=
\hat{\varphi}^+_\lambda|_{C_+}$).}
	$$
	
	Therefore from (\ref{eq65}) it follows that
	\begin{equation}\label{eq66}
		y_0\not\in\{\hat{u},v_0,0\}.
	\end{equation}
	
	Also, from the claim  we have that $u_0$ is a local minimizer of $\varphi_\lambda$.
	
	Hence
	\begin{equation}\label{eq67}
		C_k(\varphi_\lambda,u_0)=\delta_{k,0}\ZZ\ \mbox{for all}\ k\in\NN_0.
	\end{equation}
	
	Note that $d_l\geq2$ (since $l\geq m\geq2$). Therefore
	$$
	d_{l-1}\geq1
	$$
	and so from (\ref{eq65}) and (\ref{eq67}), we infer that
	$$
	y_0\neq u_0.
	$$
	
	So, we conclude that $y_0\in C^1(\overline\Omega)\backslash\{0\}$ is a fourth nontrivial solution of (\ref{eqP}) (for all $\lambda\in(0,\hat\lambda)$) distinct from $u_0,\hat{u},v_0$.
\end{proof}

If we strengthen the hypotheses on $f(z,\cdot)$ we can improve the above multiplicity theorem and produce a fifth nontrivial smooth solution.

The new conditions on the nonlinearity $f(z,x)$ are the following:

\smallskip
$H(f)'$: $f:\Omega\times\RR\rightarrow\RR$ is a measurable function such that for almost all $z\in\Omega$, $f(z,0)=0$, $f(z,\cdot)\in C^1(\RR)$, hypotheses $H(f)'(i),(ii),(iii)$ are the same as the corresponding hypotheses $H(f)(i),(ii),(iii)$ and

(iv) there exist $l\in\NN,\ l\geq m$ such that
$$
	\begin{array}{ll}
	& f'_x(z,0)=\lim_{x\rightarrow0}\frac{f(z,x)}{x}\ \mbox{uniformly for almost all}\ z\in\Omega, \\
	& f'_x(z,0)\in[\hat\lambda_l,\hat\lambda_{l+1}]\ \mbox{for almost all}\ z\in\Omega, \\
	& f'_x(\cdot,0)\not\equiv\hat\lambda_l,\ f'_x(\cdot,0)\not\equiv\hat\lambda_{l+1}.
	\end{array}
$$

\begin{theorem}\label{th13}
	If hypotheses $H(\xi),\ H(\beta),\ H(f)'$ hold, then there exists $\hat\lambda>0$ such that for all $\lambda\in(0,\hat\lambda)$ problem (\ref{eqP}) has at least five nontrivial solutions
	$$
	\begin{array}{ll}
		u_0,\hat{u}\in(-C_+),\ u_0(z)<0\ \mbox{for all}\ z\in\Omega, \\
		v_0\in C_+,\ v_0(z)>0\ \mbox{for all}\ z\in\Omega, \\
		y_0,\ \hat{y}\in C^1(\overline\Omega)\backslash\{0\}.
	\end{array}
	$$
\end{theorem}

\begin{proof}
	Now we have $\varphi_\lambda\in C^2(H^1(\Omega)\backslash\{0\},\RR)$. Similarly, $\hat\varphi^\pm_\lambda\in C^2(H^1(\Omega)\backslash\{0\},\RR)$.
	
	The solutions $u_0,\hat{u}, v_0, y_0$ are a consequence of Theorem \ref{th12}. From Proposition \ref{prop9} and (\ref{eq67}), we have
	\begin{equation}\label{eq68}
		C_k(\varphi_\lambda,u_0) = C_k(\varphi_\lambda,0)=\delta_{k,0}\ZZ\ \mbox{for all}\ k\in\NN_0,\ \lambda\in(0,\hat\lambda).
	\end{equation}
	
	Also, from the proof of Theorem \ref{th12}, we know that
	$$
	\begin{array}{ll}
		\hat{u}\ \mbox{is a critical point of}\ \hat\varphi^-_\lambda\ \mbox{of mountain pass type}, \\
		v_0\ \mbox{is a critical point of}\ \hat\varphi^+_\lambda\ \mbox{of mountain pass type}.
	\end{array}
	$$
	
	Invoking Corollary 6.102 of Motreanu, Motreanu \& Papageorgiou \cite{5}, we have
	\begin{equation}\label{eq69}
		C_k(\varphi^-_\lambda,\hat{u}) = C_k(\hat\varphi^+_\lambda,v_0) = \delta_{k,1}\ZZ\ \mbox{for all}\ k\in\NN_0.
	\end{equation}
	
	The continuity in the $C^1$-norm of the critical groups (see Theorem 5.126 in Gasinski \& Papageorgiou \cite[p. 836]{3}), implies that
	\begin{eqnarray}
		C_k(\hat{\varphi}^-_\lambda,\hat{u}) = C_k(\varphi_\lambda,\hat{u})\ \mbox{for all}\ k\in\NN_0, \label{eq70} \\
		C_k(\hat{\varphi}^+_\lambda, v_0) = C_k(\varphi_\lambda,v_0)\ \mbox{for all}\ k\in\NN_0. \label{eq71}
	\end{eqnarray}
	
	From (\ref{eq69}), (\ref{eq70}), (\ref{eq71}) it follows that
	\begin{equation}\label{eq72}
		C_k(\varphi_\lambda, \hat{u}) = C_k(\varphi_\lambda,v_0)=\delta_{k,1}\ZZ\ \mbox{for all}\ k\in\NN_0.
	\end{equation}
	
	The fourth nontrivial solution $y_0\in C^1(\overline\Omega)$ was produced by using Theorem 3.1 of Perera \cite{12}. According to that theorem, we can also find another function $\hat{y}\in H^1(\Omega)$, $\hat{y}\neq y_0$ such that
	\begin{equation}\label{eq73}
		\hat{y}\in K_{\varphi_\lambda}\subseteq C^1(\overline\Omega)\ \mbox{and}\ C_{d_l}(\varphi_\lambda, \hat{y})\neq0\ (d_l\geq2).
	\end{equation}
	
	From (\ref{eq68}), (\ref{eq72}), (\ref{eq73}) we conclude that
	$$
	\hat{y}\in C^1(\overline\Omega)\backslash\{u_0,\hat{u},v_0,y_0,0\}
	$$
	is the fifth nontrivial solution of problem (\ref{eqP}), for all $\lambda\in(0,\hat\lambda)$.
\end{proof}
\medskip
{\bf Acknowledgements.}  This research was supported by the Slovenian Research Agency grants
P1-0292, J1-8131, J1-7025, N1-0064, and N1-0083. V.D.~R\u adulescu acknowledges the support through a grant of the Romanian Ministry of Research and Innovation, CNCS--UEFISCDI, project number PN-III-P4-ID-PCE-2016-0130,
within PNCDI III.

\end{document}